\documentclass[12pt]{amsart}
\usepackage{amsmath,amsthm,amssymb}
\usepackage[all]{xy}

 \newtheorem{thm}{Theorem}[section]
 
 \newtheorem{lem}[thm]{Lemma}
 \newtheorem{prop}[thm]{Proposition}
 \theoremstyle{definition}
 \newtheorem{defn}[thm]{Definition}
 \theoremstyle{remark}
 \newtheorem{rem}[thm]{Remark}
 \newtheorem{ex}[thm]{Example}
 \newtheorem{prob}[thm]{Problem}
 \newtheorem{conj}[thm]{Conjecture}
 \numberwithin{equation}{section}

\newcommand{\thmref}[1]{Theorem~\ref{#1}}
\newcommand{\lemref}[1]{Lemma~\ref{#1}}

\newcommand{\proref}[1]{Proposition~\ref{#1}}
\newcommand{\remref}[1]{Remark~\ref{#1}}

\newcommand{\exref}[1]{Example~\ref{#1}}
\newcommand{\conjref}[1]{Conjecture~\ref{#1}}

\newcommand{\figref}[1]{Figure~\ref{#1}}
\newcommand{\sref}[1]{Section~\ref{#1}}

\DeclareMathOperator{\spec}{Spec}

\DeclareMathOperator{\supp}{Supp}

\DeclareMathOperator{\nr}{nr}
\DeclareMathOperator{\brr}{\bar{r}}

\DeclareMathOperator{\di}{div}

\DeclareMathOperator{\lcm}{lcm}

\DeclareMathOperator{\mult}{mult}

\DeclareMathOperator{\gon}{gon}

\DeclareMathOperator{\cf}{cff}
\newcommand{\m}{\mathfrak m}

\newcommand{\PP}{\mathbb P}

\newcommand{\Z}{\mathbb Z}
\newcommand{\Q}{\mathbb Q}
\newcommand{\R}{\mathbb R}
\newcommand{\C}{\mathbb C}

\newcommand{\bbZ}{\ensuremath{\mathbb Z}}

\newcommand{\cA}{\mathcal A}

\newcommand{\cF}{\mathcal F}

\newcommand{\cI}{\mathcal I}

\newcommand{\cO}{\mathcal O}

\renewcommand{\t}{\widetilde}

\renewcommand{\:}{\colon}

\newcommand{\ol}[1]{\overline {#1}}

\newcommand{\fl}[1]{\left\lfloor #1 \right\rfloor}
\newcommand{\ce}[1]{\left\lceil #1 \right\rceil}

\newcommand{\defset}[2]{{\left\{#1\,\left| \,#2 \right. \right\}}}
\newcommand{\V}{(V,p)}

\newcommand{\minpg}{\text{min-}p_g}

\newcommand{\la}[2]{\lambda^{(#1)}_{#2}}

\begin{document}
%
%
%
%
%
%
%
%
%
\title[Normal reduction numbers of normal surface singularities]
 {Normal reduction numbers of normal surface singularities}
\author[Tomohiro Okuma]{Tomohiro Okuma}

\address{Department of Mathematical Sciences \\
Yamagata University \\
 Yamagata, 990-8560 \\
Japan.}
\email{okuma@sci.kj.yamagata-u.ac.jp}

%

\thanks{This work was partially supported by JSPS Grant-in-Aid 
for Scientific Research (C) Grant Number 17K05216}



\dedicatory{Dedicated to Professor Andr\'as N\'emethi 
on the occasion of his sixtieth birthday}

\begin{abstract}
This article consists of two parts. 
The first part is a survey on the normal reduction numbers of normal surface singularities. It includes results on elliptic singularities, cone-like singularities and homogeneous hypersurface singularities.
In the second part, we prove a new results on the normal reduction numbers and related invariants of Brieskorn complete intersections.
\end{abstract}

\maketitle

\section{Introduction}

In this paper, we survey results on the normal reduction numbers of normal complex surface singularities and some related topics (\cite{o.h-ell}, \cite{OWYnrcn}, \cite{OWYrees}, \cite{OWYnrbr}), and prove a new results on the normal reduction numbers of Brieskorn complete intersections.
The normal reduction number has appeared in the study of normal Hilbert polynomials from a ring-theoretic point of view (cf. \cite{ito.Int}, \cite{MMV}).
We study the normal reduction numbers of the local ring of normal surface singularities using resolution of singularities, 
and we wish to know what kind of geometric property of singularities relates to the normal reduction numbers.

Let us briefly recall some basic facts about integral closure and reduction of ideals in a local ring. Let  $(A,\m)$ be a Noetherian local ring and  $I$ an $\m$-primary ideal (namely, $\sqrt{I}=\m$).  
Let $\overline{I}$ denote the integral closure of $I$, that is, $\ol{I}$ is an ideal of $A$ consists of all elements $z \in A$ such that  
$z^n+c_1z^{n-1}+\cdots+c_n=0$ 
for some $n \ge 1$ and $c_i \in I^i$ $(i=1,\ldots,n)$. 
The ideal $I$ is said to be integrally closed if $I=\ol{I}$.
An ideal $Q\subset I$ is called a {\em reduction} of $I$ if $I^{n+1}=QI^n$ for some $n\ge 0$. It is known that an ideal $Q$ is a reduction of $I$ if and only if $I \subset \ol{Q}$ (cf. \cite[1.2.5]{HS-book}).
For a reduction $Q$ of $I$, ${\mathrm r_Q(I)}:=\min\defset{n}{I^{n+1}=QI^n}$ is called the reduction number of $I$  with respect to $Q$. 

Let $\V$ be a normal complex surface singularity\footnote{In our papers \cite{o.h-ell}, \cite{OWYnrcn}, \cite{OWYrees}, \cite{OWYnrbr}, we treat a singularity $(\spec A, \m)$, where $(A,\m)$ is an excellent normal two-dimensional local ring such that the residue field $k$ is algebraically closed and $k\subset A$.} and $\cO_{V,p}$ the local ring of the singularity with maximal ideal $\m$.
Let $I\subset \cO_{V,p}$ be an $\m$-primary integrally closed ideal.
It is known that any minimal reduction of $I$ is generated by two elements and that two general elements of $I$ generate a minimal reduction of $I$ (see \cite[8.3.7, 8.6.6]{HS-book}). 
Suppose that $Q$ is a minimal reduction of $I$.
We define two normal reduction numbers, which are analogues of the reduction number $r_Q(I)$,  as follows:
\begin{align*}
\nr(I) &= \min\{n \in \bbZ_{\ge 0} \,|\, \overline{I^{n+1}}=Q\overline{I^n}\}, \\
\brr(I) &= \min\{n \in \bbZ_{\ge 0} \,|\, \overline{I^{N+1}}=Q\overline{I^N} \; \text{for every $N \ge n$}\}.
\end{align*}
 We note that $\nr(I)$ and $\brr(I)$ are independent of the choice of $Q$ (see e.g. \cite[Theorem 4.5]{Hun.Hilb}, \proref{p:qr}), though $r_Q(I)$ is not independent of the choice of a minimal reduction $Q$ in general. 
It is obvious by the definition that $\nr(I)\le \brr(I)$. 
We will show that $\brr(I)\le p_g\V+1$ in general (see \proref{p:qr}). 
We can also show that for any integer $g\ge 2$ there exists a singularity $\V$ with $\nr(I)=1$ and $\brr(I) =p_g\V +1=g+1$ (Example \ref{nr<br}).
We define 
\begin{align*}
\nr\V&=\max\{\nr(J) \,|\, \text{$J$ is an $\m$-primary integrally closed ideal of $\cO_{V,p}$}\}, \\
\brr\V&=\max\{\brr(J) \,|\, \text{$J$ is an $\m$-primary integrally closed ideal of $\cO_{V,p}$}\}.
\end{align*}
The invariant $\brr\V$ naturally appears in several situation as follows. 
For any  $\m$-primary integrally closed ideal $I \subset \cO_{V,p}$, 
there exist a resolution $\pi\: X \to V$ and a divisor $Z$ on $X$ such that $\cO_X(-Z)$ is $\pi$-generated and $I=\pi_*\cO_X(-Z)_p$ (see \sref{s:CC}). Let $r:=\brr(I)$.
By the definition of $\brr$ and \proref{p:qr}, we have the following:
\begin{enumerate}
\item Brian\c con-Skoda type inclusion (cf. \cite{LT}, \cite{HH}): 
$\overline{I^{r+k}} \subset Q^k$ for $k\ge 1$.
\item The natural homomorphism $\pi_*\cO_X(-nZ) \otimes \pi_*\cO_X(-Z) \to \pi_*\cO_X(-(n+1)Z)$ is surjective for $n\ge r$. 
\item The function $\phi(n):=\dim_{\C}H^0(\cO_X)/H^0(\cO_X(-nZ))$  is a polynomial function of $n$ for $n\ge r$; note that $\phi(n)=\chi(\cO_{nZ})+h^1(\cO_X)-h^1(\cO_X(-nZ))$ by Kato's Riemann-Roch Theorem (\cite{kato}). 
\end{enumerate}
So we expect that the normal reduction numbers can characterize good singularities.  For example, we see that $\V$ is a rational singularity if and only if $\brr\V=1$ (see \proref{p:rat}). 
However, we can only show that $\brr\V=2$ if $\V$ is an elliptic singularity (see \thmref{t:ellr2}, \proref{p:br=2}).  
At present, we have computed the normal reduction numbers only for some special cases, and we do not know whether those invariants are topological or not. 

This paper is organized as follows.
Sections \ref{s:CC}--\ref{s:cone} are devoted to a survey of fundamental results on the normal reduction numbers and some related topics.
In \sref{s:CC}, we set up notation and briefly recall the basic results on the cohomology groups of ideal sheaves of cycles on a resolution space. Then we mention a question about the range of the dimension of those cohomology groups.  
In \sref{s:h-nr}, we give a relation between the normal reduction numbers and the dimension of the cohomology groups associated with an $\m$-primary integrally closed ideal in $\cO_{V,p}$ and review fundamental results on the normal reduction numbers.
Then we review the results on elliptic singularities.
In \sref{s:cone}, we consider the cone-like singularities, namely, those homeomorphic to the cone over a nonsingular curve. 
We give an upper bound of $\brr$ using the genus and gonality of the curve and the self-intersection number of the fundamental cycles. 
Then we show a formula for the normal reduction numbers of homogeneous hypersurface singularities.
In \sref{s:BCI}, we prove an explicit formula for $\brr$ of the maximal ideal of a Brieskorn complete intersection and apply the formula to classify elliptic singularities, which are natural generalization of the results about Brieskorn hypersurface singularities in \cite{OWYnrbr}.

\section{Cycles and Cohomology}\label{s:CC}

Let $(V,p)$ be a normal complex surface singularity, namely, the germ of a normal complex surface $V$ at $p\in V$.
We always assume that $V$ is Stein and suitably small.
Let $\pi\: X \to V$ denote a resolution of the singularity $(V,p)$ with exceptional set $E= \pi^{-1}(p)$ and let $\{E_i\}_{i\in \cI}$ denote the set of irreducible components of $E$.
We call a divisor on $X$ supported in $E$ a {\em cycle} and denote by $\sum \Z E_i$ the group of cycles.

For a function $h\in H^0(\cO_X(-E))$, we denote by $(h)_E\in \sum\Z E_i$ the exceptional part of the divisor $\di_X(h)$; 
so, $\di_X(h)-(h)_E$ is an effective divisor containing no components of $E$. 
We simply write $(h)_E$ instead of $(h\circ \pi)_E$ for $h\in \m$.

An element of $\sum \Q E_i:=(\sum \Z E_i)\otimes \Q$ is called a {\em $\Q$-cycle}.
A $\Q$-cycle $D$ is said to be {\em nef} (resp. {\em anti-nef}) if $DE_i\ge 0$ (resp. $DE_i\le 0$) for all $i\in \cI$.
Note that if $D\ne 0$ is anti-nef, then $D\ge E$.

\begin{defn}
The {\em maximal ideal cycle} on $X$ is the minimum of $\defset{(h)_E}{h \in \m}$ and denoted by $M_X$.
There exists a $\Q$-cycle $Z_{K_X}$ such that $(K_X+Z_{K_X})E_i=0$ for every $i\in \cI$, where $K_X$ is a canonical divisor on $X$.
We call $Z_{K_X}$ the {\em canonical cycle} on $X$. 
\end{defn}

In the following, we assume that $Z>0$ is a cycle such that $\cO_X(-Z)$ {\em has no fixed component}, namely, there exists a function $h\in H^0(\cO_X(-Z))$ such that $(h)_E=Z$.
We say that $\cO_X(-Z)$ is {\em generated} if it is $\pi$-generated (i,e., $\pi^*\pi_*\cO_X(-Z)\to \cO_X(-Z)$ is surjective).
For any coherent sheaf $\cF$ on $X$, we write $H^i(\cF)=H^i(X,\cF)$ and $h^i(\cF)=\dim_{\C}(H^i(\cF))$.

\begin{defn}
The {\em geometric genus} of the singularity $\V$ is defined by $p_g\V=h^1(\cO_X)$.
\end{defn}

\begin{defn}
Let $A\ge 0$ be an effective cycle on $X$ and let
\[
h(A)=\max  \defset{h^1(\cO_{B})}{B \in \sum \Z E_i, \; B\ge 0, \; \supp (B)\subset \supp (A)}.
\]
We put $h^1(\cO_{B})=0$ if $B=0$.
There exists a unique minimal cycle $C$ such that $h^1(\cO_C)=h(A)$
(cf. \cite[4.8]{chap}). 
We call $C$ the {\em cohomological cycle} of $A$.
Note that $p_g\V=h(E)$ and that if $\V$ is Gorenstein and $\pi$ is the minimal resolution, then $Z_{K_X}$ is the cohomological cycle of $E$ (\cite[4.20]{chap}).

We define a reduced cycle $A^{\bot}$ to be the sum of the components $E_i\subset E$ such that $AE_i=0$.
\end{defn}

\begin{rem}\label{r:h1bot}
Let $F_1, \dots, F_k$ be the connected component of $Z^{\bot}$ and let $(V_i, p_i)$ be the normal surface singularity obtained by contracting $F_i$.
If $C$ is the cohomological cycle of $Z^{\bot}$, we have
\[
h^1(\cO_C)=\sum_{i=1}^kp_g(V_i,p_i).
\]
\end{rem}

\begin{defn}
Let $q(Z)=h^1(\cO_X(-Z))$ and $q_Z(n)=h^1(\cO_X(-nZ))$ for $n \ge 0$.
Let $s(Z)=\min\defset{n \in \Z_{\ge 0}}{q_Z(n)=q_Z(n+1)}$.
\end{defn}

\begin{prop}[See {\cite[\S3]{OWYrees}, \cite[3.6]{o.h-ell}}]\label{p:nZ}
We have the following. 
\begin{enumerate}
\item $q_Z(n) \ge q_Z(n+1)$ for every integer $n \ge 0$.

\item If $q_Z(1)=p_g\V$, namely, $s(Z)=0$, 
then  $q(n)=p_g\V$ for $n\ge 0$.
\item If $\cO_X(-Z)$ is generated, then 
$q_Z(n)=q_Z(s(Z))=h^1(\cO_C)$ for $n \ge s(Z)$, where $C$ is the cohomological cycle of $Z^{\bot}$. 
\item $\cO_X(-nZ)$ is generated for $n>s(Z)$.
\end{enumerate}
\end{prop}

We are interested in the range of the function $q$.
Let $\cA$ (resp. $\cA'$) denotes the set of the pairs $(Y,W)$ such that $W>0$ is a cycle on a resolution $Y\to V$ such that $\cO_Y(-W)$ is generated (resp. has no fixed components).
Clearly, $\cA\subset \cA'$.
Let 
\[
q(\cA)=\defset{h^1(\cO_Y(-W))}{(Y,W)\in \cA}, \ \
q(\cA')=\defset{h^1(\cO_Y(-W))}{(Y,W)\in \cA'}.
\]
By \proref{p:nZ}, we have 
\[
q(\cA) \subset q(\cA') \subset\{0,1,\dots, p_g\V\}.
\]
The proof of the following theorem is included in the proof of \cite[3.12]{o.h-ell}.

\begin{prop}\label{p:imgq}
We have the equality
\[
q(\cA')=\{0,1,\dots, p_g\V\}.
\]
\end{prop}

\begin{conj}\label{c:q}
For every normal complex surface singularity, 
the equality $q(\cA) = q(\cA')$ holds.
\end{conj}

At present, we have the equality $q(\cA) = q(\cA')$ only for a few cases (cf. \proref{p:ellA}, \exref{nr<br}).
Some results related to \conjref{c:q} are obtained in \cite{nagy.h1}.

The next lemma is used in \sref{s:BCI}.
For a $\Q$-cycle $D$, let $\cO_X(D)=\cO_X(\fl{D})$, where $\fl{D}$ denotes the integral part of $D$.
\begin{lem}\label{l:filt}
Let $C<E$ be a reduced cycle and $\{I_n\}_{n\in \Z_{\ge 0}}$ a filtration of $\cO_{V,p}$ such that $(h)_E\ge nC$ for all $n\in \Z_{\ge 0}$ and all $h\in I_n\setminus\{0\}$ and that $\bigoplus_{n\ge 0} I_n/ I_{n+1}$ is reduced. 
Assume that there exists an anti-nef $\Q$-cycle $\t C=\sum a_iE_i$ such that $a_i=1$ for $E_i\le C$ and $\t CE_i=0$ for every $E_i\not\le C$.
Moreover assume that there exists an integer $d>0$ such that $d\t C\in \sum \Z E_i$ and $(h)_E=d\t C$ for some $h\in I_d$.
Then $I_n=\t I_n:=\pi_*\cO_X(-n\t C)_p$.
\end{lem}
\begin{proof}
First we show that $I_n\subset \t I_n$ for every $n\ge 0$.
Let $h\in I_n$ and $\Delta=(h)_E-n\t C$. We write $\Delta=\Delta_1-\Delta_2$, where $\Delta_1$ and $\Delta_2$ are effective and have no common components. Since $(h)_E\ge n C$, by the assumption on $\t C$, we have $\supp(\Delta_2) \subset \supp(\t C-C) = \supp(E-C)$, and hence $\t C\Delta_2=0$.
If $\Delta_2\ne 0$, then $0<-\Delta_2^2\le \Delta\Delta_2=(h)_E\Delta_2$; it contradicts that $(h)_E$ is anti-nef.
Hence $\Delta=\Delta_1\ge 0$, namely, $h\in \t I_n$.

From the arguments in \S 2.2--2.4 of \cite{tki-w}, since  $\bigoplus_{n\ge 0} I_n/ I_{n+1}$ is reduced, we have a $\Q$-cycle $D>0$ such that $I_n=\pi_*\cO_X(-nD)_p$ for all $n\in \Z_{\ge 0}$,  and we may assume that $dD\in \sum\Z E_i$ and $\cO_X(-dD)$ is generated.
The inclusion $I_d\subset \t I_d$ implies that $dD \ge d\t C$. 
Since there exists $h\in I_d$ such that $d\t C=(h)_E\ge dD$, we obtain  $\t C=D$. 
\end{proof}

\section{Cohomology and normal reduction numbers}
\label{s:h-nr}

Let $\m \subset \cO_{V,p}$ denote the maximal ideal.
In the following, we always assume that $I \subset \cO_{V,p}$ is an $\m$-primary integrally closed ideal, namely,  $I$ satisfies that $\sqrt{I}=\m$ and $\bar I = I$.
Let $Q$ be a minimal reduction of $I$.
Then there exist a resolution $\pi\: X\to V$ 
and a cycle $Z>0$ such that
\[
I=I_Z:=\pi_*\cO_X(-Z)_p
\]
and  $I\cO_X=\cO_X(-Z)$ (cf. \cite[\S 6]{Li}).
In this case, we say that $I$ is {\em represented} by a cycle $Z$ on $X$.  We use the symbol ``$I_Z$'' only when $\cO_X(-Z)$ is generated.
Conversely, such an ideal ${I_Z}$ is  $\m$-primary and integrally closed.
Note that $\ol{I_ZI_{Z'}}=I_{Z+Z'}$.
Thus we can write
\begin{align*}
\nr(I_Z)&=\min\defset{n\in \Z_{>0}}{I_{(n+1)Z}=QI_{nZ}}, \\ 
 \brr(I_Z)&=\min\defset{n\in \Z_{>0}}{I_{(m+1)Z}=QI_{mZ}, \; m \ge n}.
\end{align*}

In the rest of this section, we always assume that $I$ is represented by a cycle $Z$ on $X$, namely, $I=I_Z$.

\begin{defn}
We put $q(I)=q(Z)=h^1(\cO_X(-Z))$; this is independent of the representation of $I$ (cf. \cite[Lemma 3.4]{OWYgood}). 
\end{defn}

\begin{prop}[Cf. {\cite[\S2]{OWYrees}}]\label{p:qr}
Let $q_I(n):=q(\ol{I^n})=q_Z(n)$ for $n \ge 0$.
We have the following.
\begin{enumerate}
\item 
For any integer $n \ge 1$, we have 
\[
2  q_I(n) + \dim_{\C}(\overline{I^{n+1}}/Q\overline{I^n})
=q_I(n+1)+q_I(n-1).  
\]
In particular, 
\[
\nr(I) = \min\defset{n \in \bbZ_{\ge 0}}
{q_I(n-1)-q_I(n)=q_I(n)-q_I(n+1)}.
\]
\item We have
\[
\brr(I) = \min\defset{n \in \bbZ_{\ge 0}}{q_I(n-1)=q_I(n)}.
\]
In particular, $\brr(I)=s(Z)+1 \le p_g\V+1$ and $q_I(n) = q_I(s(Z))$ for every $n\ge s(Z)$.
\end{enumerate}
\end{prop}
\begin{proof}
We write $H^i(Z):=H^i(\cO_X(-Z))$. 
Let $h_1, h_2\in H^0(Z)$ and  
$Q:=(h_1, h_2)\subset \cO_{V,p}$.
Suppose that $h_1, h_2$ are sufficiently general so that $Q$ is a minimal reduction of $I=I_Z$ and that the following sequence is exact:
\[
0 \to \mathcal{O}_X(-(n-1)Z)
\xrightarrow{(h_1\; h_2)} 
{\mathcal{O}_X(-nZ)^{\oplus 2}
\xrightarrow{\binom{-h_2}{h_1}}
\mathcal{O}_X(-(n+1)Z)} \to 0.
\]
Taking cohomology, we obtain the long exact sequence:
\[
0 \to \ol{I^n}Q \to \ol{I^{n+1}}\to H^1((n-1)Z) \to H^1(nZ)^{\oplus 2} \to H^1((n+1)Z) \to0.  
\]
This yields (1).  We write 
\[
\dim_{\C} (\ol{I^{n+1}}/Q\ol{I^n} )=\Delta_I(n-1)-\Delta_I(n) \ge 0,
\]
where ${\Delta_I(n)}=q_I(n)-q_I(n+1)$. By \proref{p:nZ} (1), $\Delta_I(n) \ge 0$.
Therefore, if $\Delta_I(n-1)=0$, then $\Delta_I(n+k)=0$ for $k\ge 0$.
Hence we have (2).
\end{proof}

By the argument similar to the proof of \proref{p:qr}, we have
\begin{prop}[{\cite[2.9]{OWYnrbr}}]
Let $r= \nr(I)$. Then
\[
r(r-1)/2 + q(r) \le p_g\V.
\]
\end{prop}

In \cite[3.13]{OWYnrbr}, the hypersurface $V=\{x^a+y^b+z^c=0\}\subset \C^3$ with $p_g(V,o)=r(r-1)/2$ are classified.

\begin{rem}
Let $X\to Y$ be the contraction of $Z^{\bot}$ (cf. \remref{r:h1bot}). 
Then we obtain that $\brr(I)-1=\min\defset{n\in \Z_{\ge 0}}{H^1(I^n\cO_Y)=0}$ (cf. \cite[3.8]{o.h-ell}).
\end{rem}

\begin{rem}
The ideal $I$ is called the {\em $p_g$-ideal} if $q(I)=p_g\V$.
It immediately follows from \proref{p:nZ} that $\brr(I)=1$ if and only if $I$ is a $p_g$-ideal.
Moreover, the following are equivalent (see \cite[3.10]{OWYgood}, \cite[4.1]{OWYrees}):
\begin{itemize}
\item $I$ is a $p_g$-ideal.
\item $\cO_{C}(-Z)\cong \cO_C$, where $C$ is the cohomological cycle of $E$. 
\item The Rees algebra $\bigoplus_{n\ge 0}I^n$ is a Cohen-Macaulay  normal domain. 
\end{itemize} 
The $p_g$-ideals have nice properties and studied in \cite{OWYgood,OWYrees,OWYcore}. 
For example, if $I$ is a $p_g$-ideal and $J$ an $\m$-primary integrally closed ideal of $\cO_{V,p}$, then $IJ=\ol{IJ}$ and $q(IJ)=q(J)$; in particular, $p_g$-ideals form a semigroup with respect to the product.
\end{rem}

The singularity $\V$ is said to be {\em rational} if $p_g\V=0$.
Rational surface singularities can be characterized in many ways 
(\cite{artin.rat}, \cite{Li}, \cite{la.rat}, \cite{nem.ratLsp}, \cite{OWYcore}).
We have also a characterization in terms of the normal reduction numbers as follows.

\begin{prop}[{\cite[1.1]{OWYnrcn}}]\label{p:rat}
The following are equivalent:
\begin{enumerate}
\item $A$ is a rational singularity.
\item Every  $\m$-primary integrally closed ideal in  $\cO_{V,p}$ is a $p_g$-ideal.
\item $\brr(A) =1$.
\item $\nr(A)=1$.
\end{enumerate}
\end{prop}

\begin{rem}\label{r:mpg}
The singularities with $\brr(\m)=1$ ($\m$ is a $p_g$-ideal in this case) have been characterized in \cite[5.2]{PWY}.
In case $\V$ is Gorenstein and $p_g\V>0$,  the condition  $\brr(\m)=1$ implies that $\V$ is an elliptic double point (see \cite[4.3]{OWYrees}, \cite[4.10]{o.h-ell}).
\end{rem}

\subsection*{Elliptic singularities}

The elliptic singularities were introduced by P.~Wagreich,
and the theory of those singularities were
 developed by Wagreich \cite{wag.ell}, H.~Laufer \cite{la.me}, M.~Reid \cite[\S 4]{chap}, S.S.-T.~Yau \cite{yau.gor, yau.hyper, yau.normal, yau.max}, M.~Tomari
 \cite{tomari.char, tomari.ell}, and A.~N\'emethi \cite{nem.ellip}, Nagy--N\'emethi \cite{nn.ell, nn.AblIII}.

Let $Z_f$ denote the {\em fundamental cycle} on $X$, namely, the minimal non-zero anti-nef cycle.
The {\em fundamental genus} ${p_f\V}$ is defined by $p_f\V=p_a(Z)=1-\chi(\cO_Z)$. 
By the Riemann-Roch formula, $p_f\V=Z_f(Z_f+K_X)/2+1$.
This is independent of the choice of a resolution, and hence a topological invariant of the singularity $\V$.

\begin{defn}
The singularity $\V$ is said to be {\em elliptic} if $p_f\V=1$. 
\end{defn}

The following are well-known:
\begin{enumerate}
\item   For any positive integer $m$, there exists an elliptic singularity $\V$ with $p_g\V=m$ (Yau \cite[\S 2]{yau.max}).
\item   For any elliptic surface singularity $(V',p')$, there exists an elliptic singularity $\V$ with $p_g\V=1$ such that $(V',p')$ and $\V$ have the same topological type (Laufer \cite[Theorem 4.1]{la.me}).
\end{enumerate}

\begin{thm}[See {\cite[\S 3]{o.h-ell}}]\label{t:ellr2}
If $\V$ is elliptic, then  $\nr\V=\brr\V=2$.
In fact, $s(W)=1$ for any $(Y,W)\in \cA'$.
\end{thm}

The point of the proof of \thmref{t:ellr2} is as follows.
Using Yau's elliptic sequences and R{\"o}hr's vanishing theorem (\cite{rohr}), we have 
\begin{prop}[cf. {\cite[3.11]{o.h-ell}}]\label{p:ellh}
If  $\V$ is elliptic and $W>0$ is a cycle on $X$ such that $\cO_X(-W)$ has no fixed component, then $h^1(\cO_X(-W))=h^1(\cO_{C_W})$, where $C_W$ is the cohomological cycle of $W^{\bot}$.
\end{prop} 
This proposition implies that $h^1(\cO_{C_{Z}})=q_Z(n)$ for $n\ge 1$ (take $W=nZ$). If $I$ is not a $p_g$-ideal, then $s(Z)=1$, and $\brr(I)=2$ by \proref{p:qr} (2).

\begin{prop}[cf. {\cite[3.12]{o.h-ell}}]\label{p:ellA}
If $\V$ is elliptic, then $q(\cA)=q(\cA')$.
\end{prop}
\begin{proof}
By \proref{p:imgq}, there exist a resolution $Y$ and cycles $W_0, \dots, W_{p_g\V}$ on $Y$ such that $q(W_i)=i$.
Since $s(W_i)=1$, \proref{p:nZ} and \ref{p:ellh} imply that $\cO_Y(-2W_i)$ is generated and $q(W_i)=q(2W_i)$.
\end{proof}

\begin{prob}
Characterize the singularities $\V$ with $\brr\V=2$.
Is the converse of \thmref{t:ellr2} true?
\end{prob}

We define a topological invariant $\minpg\V$ to be the minimum of the geometric genus $p_g$ of normal complex surface singularities homeomorphic to $\V$.
For example, if $\V$ is elliptic, then $\brr\V-1=1=\minpg\V$ by \thmref{t:ellr2} and Laufer's result mentioned above.
Let us recall that $\brr\V\le p_g\V+1$ (\proref{p:qr}).
\begin{prob}
For a normal complex surface singularity $\V$, does the inequality $\brr\V\le \minpg\V+1$ hold? 
Characterize singularities which satisfy $\brr\V= \minpg\V+1$.
\end{prob}

\section{Cone-like singularities}\label{s:cone}

If $C$ is a nonsingular projective curve over $\C$ and $D$ an ample divisor on $C$, then $V(C,D):=\spec \bigoplus_{n \ge 0} H^0(\cO_C(nD))$
is a normal surface with at most an isolated singularity at the ``vertex''(cf. \cite{p.qh}).
Such a singularity is called a {\em cone singularity}. The exceptional set of the minimal resolution of $V(C,D)$ is isomorphic to $C$ with self-intersection number $-\deg D$.
For example, if $R = \oplus_{n\ge 0} R_n$ is a two-dimensional normal graded ring generated by $R_1$ over $R_0=\C$, then $\spec R$ has a cone singularity.  

\begin{defn}\label{d:cone}
Let $\pi_0 \: X_0 \to V$ be the minimal resolution of the singularity $\V$ and $F$ the exceptional set of $\pi_0$.
We call $\V$ a {\em cone-like singularity} if $F$ consists of a unique smooth curve. Note that in this case $\V$ is homeomorphic to the cone singularity $(V(F,-F|_F), \text{vertex})$.
\end{defn}

 In the rest of this section, we always assume that $\V$ is a cone-like singularity.
Let $g$ denote the genus of the exceptional curve $F$ of the minimal resolution $\pi_0 \: X_0 \to V$ and let ${d}=-F^2$. 
Assume that $g \ge 1$.
Let $\pi\: X \to V$ be any resolution with exceptional set $E$ as in the preceding section. 
Then we have a natural morphism $X \to X_0$.
We denote by $E_0\subset X$ the proper transform of $F$; 
this is the unique irreducible exceptional curve on $X$ with positive genus.   
Note that $d=-Z_f^2$ because $F$ is the fundamental cycle on $X_0$; the number $d$ is sometimes called the degree of $\V$.

\begin{defn} \label{gon}
Let $C$ be a nonsingular projective curve. The {\em gonality} of the curve 
$C$ is the minimum of the degree of surjective morphisms 
from $C$ to $\PP^1$, and denoted by $\gon(C)$.
It is known that $\gon(F) \le \fl{(g+3)/2}$.
\end{defn}

\begin{defn} \label{Symbol}
For any $\alpha\in \R$, 
let $[[\alpha]]=\min\defset{m\in \Z}{m>\alpha}$.
For example, $[[2]]=[[5/2]]=3$.
\end{defn}

We give an upper bound for $\brr\V$ using the invariants $g$, $d$, $\gon(E_0)$.
Note that $g$ and $d$ are topological invariant of $\V$, but $\gon(E_0)$ is not.

\begin{thm}[{\cite[3.9]{OWYnrcn}}]\label{main}  
Let $\V$ be a cone-like singularity 
and let $I = I_Z$ be an $\m$-primary integrally closed ideal represented by a cycle $Z$ on the resolution $X$.  
Then we have the following.
\begin{enumerate}
\item If  $ZE_0=0$, then   $\brr(I)\le [[(2g-2)/d]]+1$.
\item If $ZE_0<0$, then $\brr(I)\le [[(2g-2)/\gon(E_0)]]+1$.
\end{enumerate}
 In particular, {$\brr\V\le [[(2g-2)/\min\{d, \gon(E_0)\}]]+1$}.
\end{thm}

For the proof we apply R{\"o}hr's vanishing theorem (see \cite[\S3]{OWYnrcn} for the details).
The following example is a special case of \cite[3.10]{OWYnrcn} (take $b=g$).

\begin{ex}\label{nr<br}
Let $C$ be a hyperelliptic curve 
with genus $g\ge 2$
 and $D_0$ a divisor on $C$ which is the pull-back of a point via the double cover $C\to \PP^1$. 
Let $D=gD_0$ and $V=\spec\bigoplus_{n \ge 0} H^0(X, \cO_C(nD))$. 
Then $C\cong F\subset X_0$.
We have $p_g\V=g$ by \cite[Theorem 5.7]{p.qh}.  

If we take  a general  element $h\in H^0(\cO_{X_0}(-F))$, then $\di_{X_0}(h)=F+H$, 
where $H$ is the non-exceptional part and $F\cap H$ consists of distinct $2g$ points $P_1, \dots, P_{2g}$.  
We may assume that $P_1+P_2\sim D_0$.
Let $\phi\:X \to X_0$ be the blowing-up with center $\{P_3, \dots, P_{2g}\}$  and let $Z=(h)_E$, the exceptional part of $\di_X(h)$.
If we put ${E_i}=\phi^{-1}(P_i)$ for $3\le i \le 2g$, then 
 ${Z}=E_0+2(E_3+\cdots+E_{2g})$.
We can see that  $\cO_X(-Z)$ is generated since a general element of $H^0(\cO_X(-2F))$ has no zero on $H$. 

Then we have $h^1(\cO_X(-(g-1)Z))\ge h^1(\cO_{E_0}(-(g-1)Z))= h^1(K_{C})=1$
 and $H^1(\cO_X(-gZ))=0$.
It follows from \proref{p:nZ} (1) and \proref{p:qr} (2) that $q_Z(n) = g -n$ for $0\le n\le g$.  
Hence we have $\brr(I_Z)=p_g\V+1=[[(2g-2)/\gon(E_0)]]+1$, 
$\nr(I_Z) =1$, $q(\cA)=q(\cA')$. 
\end{ex}

\subsection*{Homogeneous hypersurface singularities}

Assume that $V\subset \C^3$ is a hypersurface defined by a homogeneous polynomial $f\in \C[x,y,z]$ with degree $d\ge 3$ ($\deg x= \deg y =\deg z=1$) having an isolated singularity at the origin $p\in \C^3$.
Then $F\cong \{f=0\}\subset \PP^2$, $g=(d-1)(d-2)/2$.
Let $D=-F|_F$. 
Then $V=\spec \bigoplus_{n \ge 0} H^0(\cO_C(nD))$.
Since $\m=I_F$, we have 
\[
q_{F}(n)=h^1(\cO_Y(-nF))=\sum _{m\ge n}h^1(\cO_F(mD))
=\sum_{m=n}^{d-3} \binom{d-1-m}{2}
=\binom{d-n}{3}.
\]
Hence we have  $\nr(\m) = \brr(\m)  = d-1$ by \proref{p:qr}.
By the definition, $\brr\V\ge d-1$. 
On the other hand, by Namba's theorem (Max Noether's theorem) 
\cite[Theorem 2.3.1]{Namba767}, we have $\gon(F) =d-1$.
By \thmref{main}, we have 
\[
\brr\V\le [[(2g-2)/(d-1)]]+1=[[d-2-2/(d-1)]]+1=d-1. 
\]
Hence we obtain 
\begin{thm}[{\cite[4.1]{OWYnrcn}}]\label{t:Brahma}
$\nr(\m)=\brr(\m)=\nr\V=\brr\V=d-1$.
\end{thm}  

\begin{rem}[See {\cite[\S4]{OWYnrcn}}]
Suppose that $R = \oplus_{n\ge 0} R_n$ is a normal graded ring generated by $R_1$ over $R_0=\C$ and $V=\spec R$. 
Then $\m^n=\ol{\m^n}$.
Let $a(R)$ denote the $a$-invariant of $R$ (see \cite{G-W}). 
If $Q$ is a minimal reduction of $\m$ generated 
by elements of $R_1$, we can see 
\[
\m^{a(R)+2} \ne Q\m^{a(R)+1} \ \ \text{ and} \ \  \nr(\m) = a(R) + 2 = \brr(\m).
\]
If $R=\C[x,y,z]/(f)$ as above, then $a(R)=d-3$ (cf. \cite[(3.1.6)]{G-W}).
\end{rem}

\section{Brieskorn complete intersections}\label{s:BCI}

In \cite{OWYnrbr}, we obtained an explicit expression of $\brr(\m)$ for Brieskorn hypersurfaces using ring-theoretic arguments 
and gave a classification of Brieskorn hypersurfaces having elliptic singularities.
In this section, we extend these results to the case of Brieskorn complete intersections, using resolution of singularities.

In the following, we assume that $V\subset \C^m$ is a Brieskorn complete intersection define by the following $m-2$ polynomials:
\[
q_{i1}x_1^{a_1}+\cdots +q_{im}x_{m}^{a_{m}} 
\quad (q_{ij}\in \C, \quad i=3,\dots , m), 
\]
where $a_i$ are integers such that $2\le a_1\le \dots \le a_m$.
We also assume that $V$ has an isolated singularity at the origin $p\in \C^m$.
Then, since every maximal minor of the matrix $(q_{ij})$ does not vanish  (see \cite[\S 7]{SfMfd}), we may assume that
\begin{equation}\label{eq:mat}
(q_{ij})= \begin{pmatrix}
 1 & 0 & \cdots & 0 & p_1 & q_1 \\
 0 & 1 & \cdots & 0 & p_2 & q_2  \\
 \vdots & \vdots & \ddots & \vdots & \vdots & \vdots \\
 0 & 0 & \cdots & 1 & p_{m-2} & q_{m-2} 
 \end{pmatrix},
\end{equation}
where $p_i, q_i\ne 0$ and $p_iq_j\ne p_jq_i$ for $i\ne j$.

\subsection{The maximal ideal cycle, the fundamental cycle, and the canonical cycle}

We summarize the results in \cite{MO} which will be used in this section; those are a natural extension of the hypersurface case obtained by Konno and Nagashima \cite{K-N}.
In the following, we assume that $\pi\: X\to V$ is the minimal good resolution. Since $\V$ is Gorenstein, the canonical cycle $Z_{K_X}$ is an effective cycle.

We define  positive integers $\ell$, $\ell_i$, $\alpha$, $\alpha_i$, $\hat g$, $\hat g_i$, 
and $\lambda_i$ as follows\footnote{Using the notation of \cite[\S 3]{MO}, we have $l=d_m$, $\ell_i=d_{im}$, $\alpha_i=n_{im}$, $\lambda_i=e_{im}$, $\lambda_m=e_{mm}=e_m$. }:
\begin{gather*}
\ell:=\lcm(a_1, \dots, a_m), \ \ 
\ell_i:=\lcm(a_1, \dots, \hat {a_i}, \dots, a_m), \; 
\text{where $\hat {a_i}$ is omitted}, \\ 
\alpha_i:=\ell/\ell_i,  \ \
\alpha:=\alpha_1\cdots \alpha_m, \ \ 
\hat g:=a_1\cdots a_{m}/\ell, \ \ 
\hat g_i:=\hat g \alpha_i/a_i, \ \ 
\lambda_i:=\ell/a_i. 
\end{gather*}
We easily see that the polynomials $x_i^{a_i}+p_ix_{m-1}^{a_{m-1}}+q_ix_m^{a_m}$ are weighted homogeneous polynomials of degree $\ell$ with respect to the weights $(\lambda_1, \dots, \lambda_m)$.
Then the weighted dual graph of the exceptional set $E$
 is as in \figref{fig:BCIG}, where 
\[
E=E_0+\sum_{w=1}^{m}\sum_{\nu=1}^{s_{w}}
\sum_{\xi=1}^{\hat g_w}E_{w,\nu,\xi},
\]
 $g$ denotes the genus of the central curve $E_0$, $c_0=-E_0^2$, and $c_{w,v}=-E_{w,\nu,\xi}^2$ (see \cite[4.4]{MO}).

\begin{figure}[htb]
 \begin{center}
$
\xy
(-9,0)*+{E_0}; (0,0)*+{-c_0}*\cir<10pt>{}="E"*++!D(-2.0){[g]}; 
(20,31)*+{-c_{1,1} }*\cir<14pt>{}="A_1"*++!D(-2.0){E_{1,1,1}}; 
(40,31)*+{-c_{1,2} }*\cir<14pt>{}="A_2"*++!D(-2.0){E_{1,2,1}}; 
(85,31)*+{-c_{1,s_1} }*\cir<16pt>{}="A_3"*++!D(-2.0){E_{1,s_1,1}}; 
(20,14)*+{-c_{1,1} }*\cir<14pt>{}="B_1"*++!D(-2.0){E_{1,1,\hat g_1}}; (40,14)*+{-c_{1,2} }*\cir<14pt>{}="B_2"*++!D(-2.0){E_{1,2,\hat g_1}}; (85,14)*+{-c_{1,s_1} }*\cir<16pt>{}="B_3"*++!D(-2.0){E_{1,s_1,\hat g_1}}; 
(20,-14)*+{-c_{m,1} }*\cir<14pt>{}="D_1"*++!D(-2.0){E_{m,1,1}}; 
(40,-14)*+{-c_{m,2} }*\cir<14pt>{}="D_2"*++!D(-2.0){E_{m,2,1}}; 
(85,-14)*+{-c_{m,s_{m}} }*\cir<16pt>{}="D_3"*++!D(-2.0){E_{m,s_{m},1}};
(20,-31)*+{-c_{m,1} }*\cir<14pt>{}="E_1"*++!D(-2.0){E_{m,1,\hat g_{m}}};
(40,-31)*+{-c_{m,2} }*\cir<14pt>{}="E_2"*++!D(-2.0){E_{m,2,\hat g_{m}}};
(85,-31)*+{-c_{m,s_{m}} }*\cir<16pt>{}="E_3"*++!D(-2.0){E_{m,s_{m},\hat g_{m}}}; 
(57.5,31)*{\cdot },(62.5,31)*{\cdot },(67.5,31)*{\cdot }, 
(57.5,14)*{\cdot },(62.5,14)*{\cdot },(67.5,14)*{\cdot }, 
(57.5,-14)*{\cdot},(62.5,-14)*{\cdot},(67.5,-14)*{\cdot}, 
(57.5,-31)*{\cdot},(62.5,-31)*{\cdot},(67.5,-31)*{\cdot}, 
(62.5,25)*{\cdot},(62.5,22.5)*{\cdot},(62.5,20)*{\cdot},    
(55.5,2.5)*{\cdot},(55.5,-1)*{\cdot},(55.5,-4.5)*{\cdot}, 
(62.5,-20)*{\cdot},(62.5,-22.5)*{\cdot},(62.5,-25)*{\cdot}, 

\ar @{-} "E" ;"A_1"  
\ar @{-} "E" ;"B_1" 
\ar @{-} "E" ;"D_1"  
\ar @{-} "E" ;"E_1"

\ar @{-} "A_1"; "A_2"   
\ar @{-} "A_2"; (55,31) \ar @{-} (70,31);"A_3"  

\ar @{-} "B_1"; "B_2" 
\ar @{-}"B_2";(55,14) \ar @{-} (70,14);"B_3" 
\ar @/^2mm/@{-}^{\hat g_1} (92,31);(92,14) 

\ar @{-} "D_1" ; "D_2"  
\ar @{-} "D_2";(55,-14) \ar @{-} (70,-14);"D_3" 

\ar @{-} "E_1" ; "E_2" 
\ar @{-} "E_2";(55,-31) \ar @{-} (70,-31);"E_3" 
\ar @/^2mm/@{-}^{\hat g_{m}} (92,-14);(92,-31) 
\endxy
$
\end{center}
\caption{\label{fig:BCIG}}
\end{figure}

For any $\Q$-cycle $B$ on $X$ and any irreducible component $F\subset E$, let $\cf_{F}(B)$ denote the coefficient of $F$ in $B$.
Let $Z^{(i)}=(x_i)_E$.

\begin{thm}[{\cite[4.4]{MO}}]\label{t:resol}
We have the following:
$$
Z^{(i)}=\la{i}{0}E_0+\sum_{w=1}^{m}\sum_{\nu=1}^{s_{w}}
\sum_{\xi=1}^{\hat g_w}\la{i}{w,\nu,\xi}E_{w,\nu,\xi} \ \ 
(1\le i \le m), 
$$
where $\la{i}{0}$ and the sequence $\{\la{i}{w,\nu,\xi}\}$ 
are determined as follows:
\begin{align*}
& \la{i}{0}:=\la{i}{w,0,\xi}:=\lambda_i, \\
& \la{i}{w,s_w+1,\xi}:=\begin{cases}
1 & \text{if $w=i$} \\
0 & \text{if $w\ne i$} , 
\end{cases}\\
& \la{i}{w,\nu-1,\xi}
=\la{i}{w,\nu,\xi}c_{w,\nu}-\la{i}{w,\nu+1,\xi}. \\
\end{align*}
The cycle $Z^{(i)}$ is the smallest one among the cycles $Z>0$ such that $Z$ is anti-nef and $\cf_{E_0}(Z)=\lambda_{i}$
(cf. \cite[2.1]{MO}).
In particular, we have $M_X=Z^{(m)}$, since $\lambda_1 \ge \cdots \ge \lambda_m$.
\end{thm}

\begin{thm}[{\cite[5.3]{MO}}]\label{t:K}
We have 
$$
Z_{K_X}=E+\frac{(m-2)l}{\alpha}Z_0-\sum_{w=1}^{m}Z^{(w)},
$$
where $Z_0$ is the anti-nef cycle such that $\cf_{E_0}(Z_0)=\alpha$ and $Z_0(E-E_0)=0$.
\end{thm}

\begin{thm}[{\cite[5.1, 5.2, 5.4]{MO}}]\label{t:pf}
If $\lambda_m\ge \alpha$, then $Z_f=Z_0$ and 
$$
p_f(V,p)= \frac{1}{2}\alpha\left\{(m-2)\hat g-\frac{(\alpha-1)\hat g}{l}-\sum_{w=1}^{m}\frac{\hat g_w}{\alpha _w}\right\}+1.
$$
If $\lambda_m\le \alpha $, then $Z_f=M_X$ and 
$$
p_f(V,p)=\frac{1}{2}\lambda_m\left\{(m-2)\hat g
-\frac{(2\ce{\lambda_m/\alpha _{m}}-1)\hat g_{m}}{\lambda_m}
-\sum_{w=1}^{m-1}\frac{\hat g_w}{\alpha _w}\right\}+1.
$$
\end{thm}

\subsection{The normal reduction numbers}

Since $M_X=(x_m)_E$ by \thmref{t:resol}, $\cO_X(-M_X)$ has no fixed components; however, it is not generated in general.

Let $H=\di_X(x_m)-M_X$.
Then $E+H$ is simple normal crossing and the set of the base points of the linear system $|\cO_X(-M)|$ is 
an empty set or $\{t_1, \dots, t_{\hat g _{m}}\}$, where $\{t_{\xi}\}=E_{m,s_{m},\xi}\cap H$ (see \thmref{t:resol}).
Let us look in detail at a point.
Let $x,y$ be the local coordinates at $t_{\xi}\in X$ such that $E=\{x=0\}$ and $H=\{y=0\}$.
We write $\eta_i=\la{i}{m,s_{m},\xi}$ and $\delta=\eta_{m-1}-\eta_{m}$.
Then $\delta \ge 0$ and $\m\cO_{X,t_{\xi}}=(x_{m-1},x_m)=(x^{\eta_{m}}y,x^{\eta_{m-1}})=x^{\eta_{m}}(y,x^{\delta})$.

\begin{prop}[{\cite[6.4]{MO}}]
The following conditions are equivalent:
\begin{enumerate}
\item $\delta=0$
\item The base points of the linear system $|\cO_X(-M)|$ on $E$ is empty.
\end{enumerate}
If $\delta>0$, each base point can be resolved by a succession of 
$\delta$ blowing-ups at the intersection of the exceptional set and the proper transform of $H$.
\end{prop}

Let $\phi\: Y\to X$ be the minimal morphism such that $\m\cO_Y$ is invertible and let $F=\phi^{-1}(E)$.
Let $W_i=(x_i)_F$ ($i=1, \dots, m$), and let $M_Y$ denote the maximal ideal cycle on $Y$ and $H_Y$ the proper transform of $H$ on $Y$.
Then 
\begin{equation}\label{eq:Wi}
W_i=\phi^*Z^{(i)} \ \ \text{for $i\ne m$, } \ \
W_m=M_Y=\phi^*Z^{(m)}+K_{Y/X},
\end{equation}
where $K_{Y/X}=K_Y-\phi^*K_X$. 
Now, $\m$ is represented by $M_Y$ and $\ol{\m^n}=I_{nM_Y}$.
Fix an irreducible component $F_{\xi}\subset F$ intersecting $H_Y$.
For any cycle $W$ on $Y$, we write $\gamma(W)=\cf_{F_{\xi}}(W)$.
Note that $\gamma(M_Y)$ is independent of the choice of a component intersecting $H_Y$ (see \thmref{t:resol}) and 
\begin{equation}
\label{eq:eta}
\gamma(W_i)=\eta_i \ \ \text{for $i\ne m$}, \ \ 
\gamma(W_m)=\gamma(M_Y)=\eta_m+\delta=\eta_{m-1}.
\end{equation}

\begin{lem}\label{l:bar}
Let $(u_1, \dots, u_m)\in (\Z_{\ge 0})^m$.
For any  positive integer $n$, 
\[
\prod_{i=1}^m x_i^{u_i}\in \ol{\m^n} \quad \text{if and only if}  \quad
\sum_{i=1}^{m-2}\frac{u_i}{a_i}\ge \frac{n-(u_{m-1}+u_m)}{a_{m-1}}.
\]
\end{lem}
\begin{proof}
We have  $\left(\prod_{i=1}^m x_i^{u_i}\right)_F=W:=\sum_{i=1}^m {u_i}W_i$.
First we show that $\prod_{i=1}^m x_i^{u_i}\in \ol{\m^n}$ if and only if $\gamma(W)\ge \gamma(nM_Y)$.
Clearly, if  $W \ge nM_Y$, then  $\gamma(W)\ge \gamma(nM_Y)$. 
 So we show the converse.
Let $W-nM_Y=D_1-D_2$, where $D_1$ and $D_2$ are effective cycles without common components. By the assumption, $D_2$ has no components of $F$ intersecting $H_Y$. Thus $M_YD_2=0$.
Then $0\le D_1D_2-D_2^2=WD_2 \le 0$. Hence $D_2=0$.
We have proved the claim.

We have the following (see \cite[Lemma 1.2 (4)]{K-N} for the first equality):
\begin{equation}\label{eq:eta2}
\eta_i=\lambda_i/\alpha_m =\ell/a_i\alpha_m \quad (1\le i \le m-1).
\end{equation}
Then  we have
\begin{align*}
\gamma(W)- \gamma(nM_Y) &=\sum_{i=1}^{m-1}u_i\eta_i+u_{m}\eta_{m-1}
-n\eta_{m-1} \\
&=\frac{l}{\alpha_m}
\left(
\sum_{i=1}^{m-2}\frac{u_i}{a_i}+\frac{u_{m-1}+u_{m}-n}{a_{m-1}}
\right).
\end{align*}
This implies the assertion.
\end{proof}

Let $P\subset A:=\C[x_1, \dots, x_m]$ denote the ideal generated by the polynomials $\defset{x_i^{a_i}+p_ix_{m-1}^{a_{m-1}}+q_ix_m^{a_m}}{i=1, \dots, m-2}$ defining $V\subset \C^m$. For simplicity, let $P$ also denote the ideal in $\C\{x_1, \dots, x_m\}$ generated by these polynomials;  so $\cO_{V,p}=\C\{x_1, \dots, x_m\}/P$. 
We easily see the following (cf. \cite[Theorem 3.1]{nw-CIuac}).

\begin{lem}\label{l:red}
For any $1 \le i \le m$, 
the quotient ring $A/(P+(x_i))$ is reduced.
\end{lem}

\begin{prop}\label{p:mmon}
For $n\in \Z_{\ge 0}$, let $I_n\subset \cO_{V,p}$ be an ideal generated by monomials $\prod_{i=1}^m x_i^{u_i}$ such that 
\[
\sum_{i=1}^{m-2}\frac{u_i}{a_i}\ge \frac{(n/\eta_{m-1})-(u_{m-1}+u_m)}{a_{m-1}}.
\]
Then $ I_{n\eta_{m-1}}=\ol{\m^n}$ for $n\in \Z_{\ge 0}$.
In particular, $\ol{\m^n}$ is generated by monomials.
\end{prop}
\begin{proof}
First we show that $G:=\bigoplus_{n\ge 0} I_n/ I_{n+1}$ is reduced.
It follows from \eqref{eq:eta} and \eqref{eq:eta2} that the inequality is equivalent to the following (cf. the proof of \lemref{l:bar}):
\begin{equation}\label{eq:In}
\gamma\left(\left(\prod_{i=1}^m x_i^{u_i}\right)_F\right)
=\sum_{i=1}^{m-1}u_i\eta_i+u_m\eta_{m-1}\ge n.
\end{equation}
Therefore the filtration $\{I_n\}_{n\in \Z_{\ge 0}}$ is induced from the weight filtration of the power series ring $\C\{x_1, \dots, x_m\}$ with weight vector 
$(\eta_1, ... , \eta_{m-1}, \eta_{m-1})\in \Z^m$.
Let $I\subset A=\C[x_1, \dots, x_m]$ denote the ideal generated by the leading form, with respect to these weights, of the polynomials $\defset{x_i^{a_i}+p_ix_{m-1}^{a_{m-1}}+q_ix_m^{a_m}}{i=1, \dots, m-2}$.
Then $A/I$ is complete intersection and isomorphic to $G$ (cf. the proof of \cite[Theorem 2.6]{nw-CIuac}).
If $a_{m-1} = a_m$, then $G=A/P$. 
If $a_{m-1} < a_m$, then $G\cong (A/P+(x_m))[x_m]$, and thus $G$ is reduced by \lemref{l:red}.

Let $C=\sum_{\xi=1}^{\hat g_m}F_{\xi}$, the sum of the irreducible components of $F$ intersecting $H_Y$.
From \eqref{eq:In}, every $h\in I_n$ satisfies $(h)_F\ge nC$.
Now we can apply \lemref{l:filt}.
Since $\eta_{m-1}\t C=M_Y$, we obtain that 
 $I_{n\eta_{m-1}}=\ol{\m^n}$.
\end{proof}

Let $Q=(x_{m-1},x_m)\subset \cO_{V,p}$. Then $x_i^{a_i}\in Q$ for every $i$, and thus $Q$ is a minimal reduction of $\m$ (cf. \cite[8.3.6]{HS-book}).

\begin{thm}\label{t:nr}
We have the following.
\begin{enumerate}
\item
\[
\nr (\m)=\brr(\m)=\fl{a_{m-1}\sum_{i=1}^{m-2}\frac{a_i-1}{a_i}}.
\]

\item  The image of the monomials $\prod_{i=1}^{m-2} x_i^{u_i}$
such that 
\[
\sum_{i=1}^{m-2}\frac{u_i}{a_i} \ge \frac{n+1}{a_{m-1}} \quad \text{and} \quad 
0\le u_i \le a_i-1 \quad (i=1, \dots, m-2)
\]
in  the vector space $\ol{\m^{n+1}}/Q\ol{\m^n}$ form a basis.
In particular, $\dim_{\C}(\ol{\m^{n+1}}/Q\ol{\m^n})$ is a non-increasing function of $n$.
\end{enumerate}\end{thm}
\begin{proof}
Note that $Q\ol{\m^n}$ and $\ol{\m^{n+1}}$ are generated by monomials for every $n\ge 0$ by \proref{p:mmon}.
Let $N=\fl{a_{m-1}\sum_{i=1}^{m-2}(a_i-1)/a_i}$. 
First we prove that $Q\ol{\m^n}=\ol{\m^{n+1}}$ for $n\ge N$.
Let $v=\prod_{i=1}^{m} x_i^{u_i}\in \ol{\m^{n+1}}$. 
By \lemref{l:bar}, we have 
\[
\sum_{i=1}^{m-2}\frac{u_i}{a_i}\ge \frac{n+1-(u_{m-1}+u_m)}{a_{m-1}}
= \frac{n-(u_{m-1}+u_m-1)}{a_{m-1}}.
\]
Therefore, if $u_{m-1}\ge 1$ or $u_m\ge 1$, we have $v/x_{m-1}\in \ol{\m^n}$ or $v/x_{m}\in \ol{\m^n}$, and hence  $v\in Q\ol{\m^n}$.
We consider the case that $u_{m-1}=u_m=0$ and $u_i\ge a_i$ for some $1\le i \le m-2$; we may assume that $i=1$.
Then it follows that $x_1^{u_1}\in x_1^{u_1-a_1}(x_{m-1}^{a_{m-1}}, x_m^{a_m})$,
since $x_1^{a_1}+p_1x_{m-1}^{a_{m-1}}+q_1x_m^{a_m}=0$.
We show that $w_1:=(x_1^{u_1-a_1}x_{m-1}^{a_{m-1}})\prod_{i=2}^{m-2} x_i^{u_i}\in Q\ol{\m^n}$.
Let $w'=w_1/x_{m-1}=(x_1^{u_1-a_1}x_{m-1}^{a_{m-1}-1})\prod_{i=2}^{m-2} x_i^{u_i}$.
Since 
\[
\frac{u_1-a_1}{a_1}+\sum_{i=2}^{m-2}\frac{u_i}{a_i}
\ge \frac{n+1}{a_{m-1}}-1=  \frac{n-(a_{m-1}-1)}{a_{m-1}},
\]
we have $w'\in \ol{\m^n}$ by \lemref{l:bar}.
Thus $w_1=x_{m-1}w'\in Q\ol{\m^n}$.
In a similar way, we also have that $w_2:=(x_1^{u_1-a_1}x_{m}^{a_{m}})\prod_{i=2}^{m-2} x_i^{u_i}\in Q\ol{\m^n}$, since $\frac{n-(a_{m-1}-1)}{a_{m-1}}\ge \frac{n-(a_{m}-1)}{a_{m-1}}$.
Hence we obtain that $v \in (w_1, w_2) \subset Q\ol{\m^n}$.
Next assume that $u_{m-1}=u_m=0$ and $u_i<a_i$ for $1\le i \le m-2$. 
Then we have
\[
\sum_{i=1}^{m-2}\frac{a_i-1}{a_i} \ge  \sum_{i=1}^{m-2}\frac{u_i}{a_i} 
\ge \frac{n+1}{a_{m-1}}.
\]
However this implies that $n \le N-1$.
Hence we obtain that $Q\ol{\m^n}=\ol{\m^{n+1}}$ for $n\ge N$.

Next we prove that $Q\ol{\m^{N-1}}\ne \ol{\m^{N}}$.
Let $v:=\prod_{i=1}^{m-2} x_i^{a_i-1}$. Then $v\not\in Q$,
because $\cO_{V,p}/Q=\C\{x_1, \dots, x_m\}/(x_1^{a_1}, \dots, x_{m-2}^{a_{m-2}}, x_{m-1}, x_m)$.
However, since 
\[
\sum_{i=1}^{m-2}\frac{a_i-1}{a_i} \ge \frac{N}{a_{m-1}},
\]
we have $v\in \ol{\m^N}$ by \lemref{l:bar}. Hence we obtain that $\brr(\m)=N$.

From the arguments above, we see that (2) holds, because 
any non-trivial linear combinations of those monomials is not in the ideal $P+(x_{m-1}, x_m)=(x_1^{a_1}, \dots, x_{m-2}^{a_{m-2}}, x_{m-1}, x_m)$.
Since $\dim_{\C}(\ol{\m^{n+1}}/Q\ol{\m^n})$ is  a non-increasing function of $n$, we have $\nr(\m)=\brr(\m)$ (cf. \proref{p:qr}).
\end{proof}

\begin{ex}
If $m=3$, we have
\begin{align*}
\nr(\m)&=\fl{\frac{a_2(a_1-1)}{a_1}}, \\
\dim_{\C}(\ol{\m^{n+1}}/Q\ol{\m^n})
&=\sharp \defset{u\in \Z}{\frac{a_1(n+1)}{a_2}\le u \le a_1-1}\\
&=\max \left(a_1-\ce{\frac{a_1(n+1)}{a_2}}, 0\right).
\end{align*}
\end{ex}

The formula for $q(\m)$ in \cite[3.8]{OWYnrbr} is generalized as follows.

\begin{prop}
Let $p(n+1)=\dim_{\C}(\ol{\m^{n+1}}/Q\ol{\m^n})$ and $q(n)=h^1(\cO_Y(-nM_Y))$  for $n\ge 0$. Then we have the following:
\[
q(n)=p_g\V+\frac{n}{2}(M_Y^2-M_YK_Y)+\sum_{i=1}^n(n+1-i)p(i).
\]
(Note that the same formula holds for any normal surface singularity.)
\end{prop}
\begin{proof}
It is well-known that the multiplicity of $\cO_{V,p}$ coincides with $\dim_{\C}\cO_{V,p}/Q$ (e.g., \cite[11.2.2]{HS-book}).
Thus we have $p(1)=\dim_{\C}(\m/Q)=-M_Y^2-1$.
From the exact sequence
\[
0\to \cO_Y(-M_Y) \to \cO_Y\to \cO_{M_Y} \to 0,
\]
we have 
\[
q(1)-q(0)=\chi(\cO_{M_Y})-1=\chi(\cO_{M_Y})+M_Y^2+p(1)=\frac{1}{2}(M_Y^2-M_YK_Y)+p(1).
\]
For $n\ge 1$, it follows from \proref{p:qr} (1) 
that 
\[
q(n)-q(n-1)=q(1)-q(0)+\sum_{i=2}^{n}p(i)
=\frac{1}{2}(M_Y^2-M_YK_Y)+\sum_{i=1}^{n}p(i).
\]
Hence we obtain
\[
q(n)-q(0)=\frac{n}{2}(M_Y^2-M_YK_Y)+\sum_{i=1}^{n}(n+1-i)p(i).
\qedhere
\]
\end{proof}

\begin{rem}
The invariant $M_Y^2-M_YK_Y$ can be computed from $a_1, \dots, a_m$ as follows.
First we have $M_Y^2=-\mult\V=-\prod_{i=1}^{m-2}a_i$  (see \cite[6.3]{MO}).
On the other hand, from \eqref{eq:Wi}, we have 
\[
M_Y^2+M_YK_Y=M_X^2+M_XK_X+2(K_{Y/X})^2=2p_a(M_X)-2-2\delta\hat g_m.
\]
We have seen a formula for $p_a(M_X)$ in \thmref{t:pf}.
\end{rem}

\subsection{Elliptic singularities of Brieskorn type}\label{ss:ellBCI}

We classify the exponents $(a_1, \dots, a_m)$ such that $\V$ is elliptic, applying the formula for $\brr(\m)$.

\begin{thm}\label{t:BCIell}
$\V$ is elliptic if and only if  $(a_1, \dots, a_m)$ is one of the following.
\begin{enumerate}
\item $(a_1, a_2, a_3)=(2,3,a)$, $a\ge 6$.
\item $(a_1, a_2, a_3)=(2,4,a)$, $a\ge 4$.
\item $(a_1, a_2, a_3)=(2,5,a)$, $5\le a \le 9$.
\item $(a_1, a_2, a_3)=(3,3,a)$, $a\ge 3$.
\item $(a_1, a_2, a_3)=(3,4,a)$, $4\le a \le 5$.
\item $(a_1, a_2, a_3, a_4)=(2,2,2,a)$, $a\ge 2$.
\end{enumerate}
\end{thm}
\begin{proof}
For the case (1)--(6) in the theorem, 
we can check that $\alpha\ge \lambda_m$ and obtain $p_f\V=1$ 
using \thmref{t:pf}.

Assume that $\V$ is elliptic. By \thmref{t:ellr2} and \thmref{t:nr}, we have
\[
3> a_{m-1}\sum_{i=1}^{m-2}\frac{a_i-1}{a_i} 
\ge a_{m-1}(m-2)/2 \ge m-2.
\]
Hence $m\le 4$.
We first consider the case $m=4$.
We have $a_3<3$, and thus $a_1=a_2=a_3=2$.
Then $\alpha/ \lambda_4=a_4/2\ge 1$ and $p_f\V=1$ by \thmref{t:pf}, 

Next assume that $m=3$.
Then we have 
 $\brr(\m)=\fl{\dfrac{(a_1-1)a_2}{a_1}}\le 2$, and thus  $(a_1-1)(a_2-3)\le 2$. 
If $a_2=2$, then $a_1=2$ and $\V$ is a rational.
Hence $a_2\ge 3$ and the list of $(a_1, a_2)$ is as follows: 
\[
(2,3), (2,4), (2,5), (3,3), (3,4).
\]
We can see that $\alpha\ge \lambda_3$ for those cases. So it follows from \thmref{t:pf} that
\[
p_f\V=\frac{1}{2}\left\{a_1a_2-a_1-a_2-(2\ce{\lcm(a_1,a_2)/a_3}-1)\gcd(a_1,a_2)\right\}+1.
\]
Let us look at each case.
\begin{enumerate}
\item The case where $(a_1,a_2)=(2,3)$.
We know that $\V$ is rational if $a_3 \le 5$.
Hence $a_3 \ge 6$. 
We have $\alpha/\lambda_3=a_3/\gcd(6,a_3)$ and $p_f\V=1$.

\item The case where $(a_1,a_2)=(2,4)$, $a_3\ge 4$.
We have $\alpha/\lambda_3=a_3/4$ if $4 \mid a_3$,  $\alpha/\lambda_3=a_3/2$ otherwise, and  $p_f\V=1$.

\item The case where $(a_1,a_2)=(2,5)$, $a_3\ge 5$.
We have $\alpha/\lambda_3=a_3/\gcd(10,a_3)$ and  $p_f\V=3-\ce{10/a_3}$.
Sine $p_f\V=1$, we have $a_3 \le 9$.

\item The case where $(a_1,a_2)=(3,3)$, $a_3\ge 3$.
We have $\alpha/\lambda_3=a_3/3$  and  $p_f\V=1$ for all $a_3 \ge 3$.

\item The case where $(a_1,a_2)=(3,4)$, $a_3\ge 4$.
We have $\alpha/\lambda_3=a_3/\gcd(12,a_3)$ and 
 $p_f\V=4-\ce{12/a_3}$.
Sine $p_f\V=1$, we have $a_3 \le 5$.
\end{enumerate}
Hence we have proved the theorem.
\end{proof}

From the proof of \thmref{t:BCIell}, we obtain that 
$\brr(\m)=2$ and  $p_f\V\ge 2$ if 
$(a_1,a_2,a_3)=(2,5,a)$ with $a \ge 10$ or $(3,4,a)$ with $a \ge 6$.
For the cases $(2,5,a)$ with $a \ge 10$ and $(3,4,a)$ with $a \ge 8$, 
letting  $Q=(y,z^2)$ and $I=\ol{Q}$, we have $\brr(I)\ge 3$.
Hence we obtain the following.

\begin{prop}[{\cite[4.5]{OWYnrbr}}]\label{p:br=2}
$\brr\V=2$ if and only if $p_f\V=1$, except for the cases $(a_1, \dots, a_m)=(3,4,6)$, $(3,4,7)$.
\end{prop}

For the reader's convenience, we put some information about the two exceptional cases above.  
Both singularities have $p_g=3$ and $p_f=2$.
The weighted dual graph $\Gamma_1$ (resp. $\Gamma_2$) of $\{x^3+y^4+z^6=0\}$ (resp. $\{x^3+y^4+z^7=0\}$) is as in \figref{fig:G}.
\begin{figure}[htb]
$
\xy
 (-18,0)*+{\Gamma_1:}; 
 (0,0)*+{-2}*\cir<9pt>{}="E"*++!D(-1.5){[1]}; 
(10,0)*+{-2 }*\cir<9pt>{}="A_1";
(-10,0)*+{-2 }*\cir<9pt>{}="B_1"; 
(0,10)*+{-2}*\cir<9pt>{}="D_1"; 
\ar @{-} "E" ;"A_1"  
\ar @{-} "E" ;"B_1" 
\ar @{-} "E" ;"D_1"  
\endxy
$
\hspace{1cm}
$
\xy
 (-28,0)*+{\Gamma_2:}; 
 (0,0)*+{-2}*\cir<9pt>{}="E"; 
(10,0)*+{-2 }*\cir<9pt>{}="A_1";
(20,0)*+{-2 }*\cir<9pt>{}="A_2"; 
(30,0)*+{-2 }*\cir<9pt>{}="A_3"; 
(-10,0)*+{-2 }*\cir<9pt>{}="B_1"; 
(-20,0)*+{-2 }*\cir<9pt>{}="B_2"; 
(0,10)*+{-2}*\cir<9pt>{}="D_1"; 
(10,10)*+{-4}*\cir<9pt>{}="D_2"; 

\ar @{-} "E" ;"A_1"  
\ar @{-} "E" ;"B_1" 
\ar @{-} "E" ;"D_1"  

\ar @{-} "A_1"; "A_2"   
\ar @{-} "A_2"; "A_3"   
\ar @{-} "B_1"; "B_2" 
\ar @{-} "D_1" ; "D_2"  
\endxy
$
\caption{}\label{fig:G}
\end{figure}

As we have seen above, the equality $\brr\V=\brr(\m)$ does not hold in general (see also \remref{r:mpg}).

\begin{prob}
For a given normal surface singularity $\V$, characterize $\m$-primary integrally closed ideals $I\subset \cO_{V,p}$ (or, cycles which represent $I$) such that $\brr\V=\brr(I)$.
Characterize normal surface singularities $\V$ 
such that $\brr\V=\brr(\m)$.
\end{prob}

%
%

\providecommand{\bysame}{\leavevmode\hbox to3em{\hrulefill}\thinspace}
\providecommand{\MR}{\relax\ifhmode\unskip\space\fi MR }
\providecommand{\MRhref}[2]{%
  \href{http://www.ams.org/mathscinet-getitem?mr=#1}{#2}
}
\providecommand{\href}[2]{#2}

\end{document}